\documentclass{amsart}

\usepackage{amsmath,amsxtra,amsthm,amssymb,tikz,tikz-qtree,xr,enumerate,stmaryrd}
\usepackage[all]{xy}
\newtheorem{theorem}{Theorem}[section]
\newtheorem{lemma}[theorem]{Lemma}
\newtheorem{prop}[theorem]{Proposition}
\newtheorem{corollary}[theorem]{Corollary}

\newtheorem*{thm}{Theorem}

\numberwithin{equation}{section}

\newcommand{\Gal}{\operatorname{Gal}}
\newcommand{\rank}{\operatorname{rank}}

\newcommand{\ord}{\operatorname{ord}}
\newcommand{\QQ}{\mathbb{Q}}

\newcommand{\cG}{\mathcal{G}}

\newcommand{\p}{\mathfrak{p}}
\newcommand{\ZZ}{\mathbb Z}

\newcommand{\Zp}{\ZZ_p}

\renewcommand{\c}{\mathrm{c}}

\newcommand{\MHG}{\mathfrak{M}_{H}(G)}
\newcommand{\M}{\mathcal{M}}
\newcommand{\X}{\mathcal{X}}
\newcommand{\Y}{\mathcal{Y}}
\newcommand{\Kinf}{ K_{\infty,\infty}}

\newcommand{\Ynm}{\Y_{n,m}}
\newcommand{\Xnm}{\X_{n,m}}
\newcommand{\enm}{e_{n,m}}
\newcommand{\Bnm}{B_{n,m}}
\newcommand{\Cnm}{C_{n,m}}
\newcommand{\Gnm}{G_{n,m}}
\newcommand{\Knm}{K_{n,m}}

\newcommand{\Inm}{I_{\nu_{n,m}}}
\newcommand{\Ipnm}{I_{\p_{n,m}}}
\newcommand{\bInm}{\bar I_{\nu_{n,m}}}
\newcommand{\IGnm}{I_{\Gnm}}

\newcommand{\cS}{\mathcal{S}}
\newcommand{\enn}{e_{n,n}}
\newcommand{\vpi}{\varpi}
\newcommand{\tX}{\tilde{\X}}
\newcommand{\tenn}{\tilde{e}_{n,n}}

 \title{Estimating class numbers over metabelian extensions}

     \author{Antonio Lei}
     \address{Antonio Lei, D\'epartement de math\'ematiques et de statistique, Universit\'e Laval, Pavillon Alexandre-Vachon, 1045 avenue de la M\'edecine, Qu\'ebec QC, Canada G1V 0A6}
     \email{antonio.lei@mat.ulaval.ca}

     \thanks{The author's research is supported by FRQNT's \'Etablissement de nouveaux chercheurs universitaires program 188809}

     \keywords{Non-commutative Iwasawa theory, class numbers, metabelian extensions}
     \subjclass[2010]{11R29 (primary), 11R23 (secondary)}

\begin{document}

\begin{abstract}
    Let $p$ be an odd prime and $\Kinf/K$ a  $p$-adic Lie extension whose Galois group is of the form $\Zp^{d-1}\rtimes \Zp$.  Under certain assumptions on the ramification of $p$ and the structure of an Iwasawa module associated to  $\Kinf$, we study the asymptotic behaviours of the size of the $p$-primary part of the ideal class groups over certain finite subextensions  inside $\Kinf/K$. This generalizes the classical result of Iwasawa and Cuoco-Monsky in the abelian case and gives a more precise formula than a recent result of Perbet in the non-commutative case when $d=2$.
\end{abstract}
 \maketitle
 
\section{Introduction}

\subsection{Setup and notation}
We fix throughout this article an odd prime $p$ and an integer $d\ge2$. Let $K$ be a number field that admits a unique prime $\p$ lying above $p$. Let $\Kinf$ be a $d$-dimensional $p$-adic Lie extension of $K$ in which only finitely many primes of $K$ ramify and  $\p$ is totally ramified. Furthermore, we fix a $\Zp$-extension $K^\c/K$ contained inside $\Kinf$ and assume that 
\begin{itemize}
\item $Gal(\Kinf/K^\c)$ is torsion-free and abelian;
\item   Every prime of $K$ that ramifies in $\Kinf$ decomposes into finitely  many primes in $K^c$.
\end{itemize}   For example, when $d=2$, we may take $K=\QQ(\mu_p)$, $K^\c=\QQ(\mu_{p^\infty})$ and $\Kinf$ the Kummer extension $\QQ\left(\mu_{p^\infty},\sqrt[p^\infty]{\alpha}\right)$,
where $\alpha\ne 0$ is an integer such that $p|\alpha$ or $p|| \alpha^{p-1}-1$ (such $\alpha$ is said to be \textit{amenable} for $p$ and this ensures that $p$ is totally ramified in $\Kinf$, see for example \cite[Proposition~2.4(i)]{lee} or \cite[Theorem~5.2 and Lemma~6.1]{viviani}). For a general $d$, we may take $\Kinf$ to be the multi-Kummer extension 
\[
\QQ\left(\mu_{p^\infty},\sqrt[p^\infty]{\alpha_1},\ldots, \sqrt[p^\infty]{\alpha_{d-1}}\right),
\] 
where $\alpha_i\ne 0$ are integers whose images  in $\QQ_p^\times/(\QQ_p^\times)^{p}$ are linearly independent over $\mathbb{F}_p$ and that the products $\alpha_1^{n_1}\cdots \alpha_{d-1}^{n_{d-1}}$, $0\le n_i\le p-1$, are all amenable (as this implies that $p$ totally ramifies in all cyclic sub-extensions of $\Kinf/K^c$).

We denote the Galois groups $G=\Gal(K_{\infty,\infty}/K)$, $H=\Gal(K_{\infty,\infty}/K^\c)$ and $\Gamma=\Gal(K^\c/K)$. In particular, we have the isomorphisms $G\cong H\rtimes \Gamma$, $H\cong\Zp^{\oplus(d-1)}$ and $\Gamma\cong\Zp$.

For $0\le m,n\le \infty$, we denote $H_m=H^{p^m}$, $\Gamma_n=\Gamma^{p^n}$, $\Gnm=H_m\rtimes \Gamma_n\le G$ and $\Knm=\Kinf^{\Gnm}$. We define $\Xnm$ to be the Hilbert $p$-class group of $\Knm$ and write $\enm$ for the $p$-exponent of $\#\Xnm$.

For any $p$-adic Lie group $L$, we shall write $\Lambda(L)$ for the Iwasawa algebra
\[
\Zp[[ L]]=\varprojlim_{N\unlhd_o L} \Zp[L/N].
\]
If $\cG$ is a pro-$p$ group and $x_1,\ldots,x_r\in\cG$, we shall write $\langle x_1,\ldots,x_r\rangle$ for the $p$-adic completion of the subgroup generated by the elements $x_1,\ldots,x_r$. Similarly, if $X_1,\ldots,X_r\subset \cG$, $\langle X_1,\ldots, X_r\rangle$ denotes the closed subgroup generated by $X_1,\ldots,X_r$.

\subsection{Main results}

Let $\X$ (respectively $\X'$) be the Galois group of the maximal abelian pro-$p$ extension of $\Kinf$ that is unramified everywhere (respectively unramified outside $p$). When $\Kinf/K$ is a $\Zp$-extension, a classical result of Iwasawa \cite{iwasawa} says that $\X$ is torsion over $\Lambda(\Gamma)$. Our first result is a generalization of this result.

\begin{thm}[Theorem~\ref{thm:torsion}]
The $\Lambda(G)$-module $\X$ is torsion.
\end{thm}

On studying the structure of $\X$ as a $\Lambda(H)$-module, we shall prove an asymptotic formula for $\enm$ with $n$ fixed and $m\rightarrow\infty$. 
\begin{thm}[Corollary~\ref{cor:growthH}]
For a fixed integer $n$, there exist integers $\mu_n$ and $\lambda_n$ such that 
\[
\enm= \mu_n\times p^{(d-1)m}+\lambda_n\times mp^{(d-2)m}+O(p^{(d-2)m})
\]
for $m\gg0$.
\end{thm}

In other words, this gives us the asymptotic growth of the class numbers in the $H$-direction. In the example above, this tells us how the size of the $p$-primary part of the ideal class group of  the  extension
\[
\QQ\left(\mu_{p^n},\sqrt[p^m]{\alpha_1},\cdots,\sqrt[p^m]{\alpha_{d-1}}\right)
\] 
varies as $m\rightarrow \infty$. In particular, these extensions are not Galois in general.
This is analogous to the main result of \cite{cuocomonsky} for Galois extensions of number fields whose Galois groups are isomorphic to direct sums of $\Zp$. In fact, our proof relies heavily on the analysis of torsion $\Lambda(H)$-modules in Cuoco and Monsky's work.

Let $\MHG$ be the category of finitely generated $\Lambda(G)$-modules $M$ such that $M/M(p)$ is finitely generated over $\Lambda(H)$, where $M(p)$ denotes the submodule of $\Zp$-torsions inside $M$. 
In non-commutative Iwasawa theory studied by Coates, Fukaya, Kakde, Kato, Ochi, Ritter, Sujatha, Venjakob, Weiss and many others (c.f. \cite{CKFVS,kakde,ochivenj,RW,venjakob02,venjakobcompo}), $\X'$ is conjectured to be inside $\MHG$ for totally real fields. Since $\X$ is a quotient of $\X'$, this would imply that $\X\in\MHG$ as well. When $K^c/K$ is the cyclotomic $\Zp$-extension, Iwasawa \cite{iwasawa,iwasawa73} conjectured that the $\mu$-invariant associated to this extension  vanishes (this is a theorem of Ferrero-Washington \cite{FW} when $k/\QQ$ is an abelian extension). This conjecture turns out to be equivalent to $\X$ itself (not just $\X/\X(p)$) being finitely generated over $\Lambda(H)$ (see Theorem~\ref{thm:MHG} below as well as  \cite[Lemma~3.3]{limfine} and \cite[Lemma~3.2]{CSfine} for the same result in different settings).

Our second result is an asymptotic formula for $e_{n,n}$ as $n\rightarrow\infty$ when $d=2$ and $\X$ is finitely generated over $\Lambda(H)$. 

\begin{thm}[Corollary~\ref{cor:main}]
Suppose that $d=2$ and  $\X$ is finitely generated over $\Lambda(H)$.  If the unique prime of $K$ above $p$ is totally ramified in $\Kinf$,  then 
\[
e_{n,n}=\tau\times np^{n}+O(p^{n}),
\]
where $\tau=\rank_{\Lambda(H)}\X$.
\end{thm}
We remark that our theorem \textit{always} applies when $\Kinf=\QQ(\mu_{p^\infty},\sqrt[p^\infty]{\alpha})$ for some integer  $\alpha$ that is amenable for $p$ since the theorem of Ferrero-Washington tells us that our hypothesis on $\X$ holds. In particular, it confirms  the prediction made by Venjakob \cite[\S8]{venjakobweier} for the extension $\QQ(\mu_{p^\infty},\sqrt[p^\infty]{p})$.
Our result can also be seen as a generalization of the classical result of Iwasawa \cite{iwasawa} on $\Zp$-extensions in the special case that the $\mu$-invariant vanishes. If $G$ is abelian, that is $G\cong\Zp^2$, we recover the main result of \cite{cuocomonsky} again in the case when the $\mu$-invariant is $0$ (denoted by $m_0$  in \textit{loc. cit.}).

In \cite{perbet}, Perbet studied the variation of class numbers when $G$ is a general $d$-dimensional $p$-adic Lie group with no $p$-torsion, without our assumption on $\X$ being finitely generated over $\Lambda(H)$ nor any assumption on the ramification of $p$. More precisely, if $\tenn$ denotes the $p$-exponent of $\#\X_{n,n}/p^n$ (rather than $\X_{n,n}$ itself), Perbet showed that
\begin{equation}\label{eq:perbet}
\tenn= \rho\times n p^{dn}+\mu\times p^{dn}+O(np^{(d-1)n}),
\end{equation}
where $\rho=\rank_{\Lambda(G)}\X$ and $\mu$ is the $\mu$-invariant of $\X$ as defined in \cite{venjakob02}. Under our assumption that $\X$ is finitely generated over $\Lambda(H)$, both $\rho$ and $\mu$ vanish. In this case, the formula of Perbet becomes simply $O(np^{(d-1)n})$. Our formula in Corollary~\ref{cor:main} is therefore slightly more precise. We shall show at the end of this article that our method yields an upper bound of $\tenn$  in the case $d=2$ and $\X\in \MHG$ (the constant $\rho$ would be $0$, but $\mu$ may be non-zero).

\begin{thm}[Corollary~\ref{cor:upper}]
If $d=2$ and $\X\in\MHG$, then 
\[
\tenn\le \mu\times p^{2n}+\tau\times np^n+O(p^n),
\]
where $\tau=\rank_{\Lambda(H)}\X/\X(p)$.
\end{thm}

\subsection*{Acknowledgment}
We would like to thank Daniel Delbourgo, Dohyeong Kim and Bharathwaj Palvannan for very informative discussions during the preparation of this paper. We are also indebted to the anonymous referees for their valuable comments and suggestions which led to many improvements in the paper.

\section{Preliminary results}

\subsection{Ramification groups and class groups}\label{S:ramification}
 Let $\Sigma$ be the set of primes of $K$ that ramify in $\Kinf$. In particular, $\p\in\Sigma$. Since $\p$ is assumed to be totally ramified in $\Kinf$, its (unique) decomposition group inside $G$ is $G$ itself.

If $\nu\in\Sigma$, we have assumed that there are only finitely many primes in $K^\c$ lying above $\nu$.  On replacing $K$ by $K_{n,0}$ if necessary, we may assume that $\nu$ is inert in $K^\c$. In particular, if $\nu_1$ and $\nu_2$ are two primes of $\Kinf$ lying above $\nu$, then they differ by an element in $H$.

Let $\M$ be the maximal  unramified abelian pro-$p$ extension of $\Kinf$ and write $\X=\Gal(\M/\Kinf)$ and $\Y=\Gal(\M/K)$. 
Note that $\X$ is normal in $\Y$ with $\Y/\X\cong G$. For each $g\in G$, let $\tilde{g}\in\Y$ be a lifting of $g$. If $x\in\X$, we have the action $x^g=\tilde{g}^{-1}x\tilde{g}$. This turns $\X$ into a $\Lambda(G)$-module. We recall from \cite[Proposition~3.1]{perbet} that $\X$ is a finitely generated $\Lambda(G)$-module.

For each $\nu\in\Sigma$, we fix $\bar \nu$ a prime of $\M$ above $\nu$ and write $I_\nu$ for the inertia group of $\bar \nu$ inside $\Y$. We note that $I_\p$ is isomorphic to $G$ since we assume that $\p$ is totally ramified in $\Kinf$ and it is unramified in $\M$. In particular, we have the isomorphism
\begin{equation}\label{eq:semi}
\Y\cong\X\rtimes G,
\end{equation}
where we identify $G$ with $I_\p$. Each element of $\Y$ may be written as $(x,g)$ for some $x\in\X$ and $g\in G$. Note in particular that under this identification, we have the equality
\begin{equation}\label{eq:Ip}
I_{\p}=\{(1,g):g\in G\}.
\end{equation}

 For $0\le m,n\le \infty$, we define $\Y_{n,m}=\Gal(\M/K_{n,m})$. For each $\nu\in\Sigma$, we write $\Inm$ for the inertia group of our choice of $\bar\nu$ inside $\Ynm$ and $\bInm$ its image under the natural projection $\Ynm\rightarrow\Gnm$. We note that $\bInm\cong\Inm$ since the extension $\M/\Kinf$ is unramified.
 
 Since $\X$ is normal in $\Y$, it is also normal in $\Ynm$. Consequently, $[\X,\Ynm]$ is normal in $\Ynm$ and we may consider the quotient $\Ynm/[\X,\Ynm]$.

 \begin{lemma}\label{lem:normalized}
The image of $\Ipnm$ in the quotient $\Ynm/[\X,\Ynm]$ is normal. That is,
\[
\langle [\X,\Ynm],\Ipnm\rangle/[\X,\Ynm]\unlhd \Ynm/[\X,\Ynm].
\]
\end{lemma}
\begin{proof}
Let $(1,g)\in \Ipnm$ (which makes sense thanks to \eqref{eq:Ip}) and $(x,h)\in\Ynm$. Then 
\[
(x,h)^{-1}(1,g)(x,h)=(x^{(g-1)h^{-1}},h^{-1}gh).
\]
Note that $x^{g-1}=x^{-1}\tilde g^{-1}x\tilde g\in[\X,\Ynm]$, hence the result.
\end{proof}

 Let  $C_{n,m}$ the subgroup of $\Y_{n,m}$ generated by $[\Y_{n,m},\Y_{n,m}]$ and all the inertia groups $\Inm^\sigma$, for $\nu\in\Sigma$ and $\sigma\in \X\rtimes H$. This contains all the inertia groups inside $\Ynm$ since any two primes of $\M$ lying above $\nu$ differ by an element in $\X\rtimes H$. 
  Finally, we define $\Bnm=\Cnm\cap\X$.  Recall from the introduction that $\Xnm$ is defined to be the Hilbert $p$-class group of $\Knm$. It may be described as follows.

 \begin{lemma}\label{lem:CFT}
We have the isomorphism $\Xnm\cong\X/\Bnm$.
 \end{lemma}
 \begin{proof}
Class field theory tells us that 
 \[
 \Xnm\cong\Ynm/\Cnm.
 \]
 By the isomorphism theorem, we have  $\X/\Bnm\cong\X\Cnm/\Cnm$. This gives the short exact sequence
 \[
 1\rightarrow\X/\Bnm\rightarrow\Xnm\rightarrow\Ynm/\X\Cnm\rightarrow 1.
 \]
 Recall that $\Ynm/\X\cong \Gnm$, the last term of the short exact sequence can be described by
 \[
 \Ynm/\X\Cnm\cong \Gnm/{\langle[\Gnm,\Gnm],\bInm^\sigma:\nu\in\Sigma,\sigma\in H\rangle}.
 \]
But $\p\in\Sigma$ and $\bar I_{\p_{n,m}}=\Gnm$ since $\p$ is totally ramified in $\Kinf$. Hence, this quotient is trivial and the result follows.
 \end{proof}

In particular, this gives us the following short exact sequence:
\begin{equation}\label{eq:SES}
0\rightarrow \Bnm/\IGnm\X \rightarrow\X/\IGnm\X\rightarrow\Xnm\rightarrow0.
\end{equation}
\subsection{Description of $\Bnm$}\label{S:estimateBnm}
We write $\IGnm$ for the augmentation ideal of $\Gnm$ in $\Lambda(G)$, that is the ideal generated by $g-1$, $g\in\Gnm$. We have the following description.

\begin{lemma}\label{lem:augmentcommut}
We have the equality $$[\X,\Ynm]=\IGnm\X.$$
\end{lemma}
\begin{proof}
Let $x\in\X$ and $y\in\Ynm$. We write $\bar y$ for the image of $y$  in $\Gnm$. Then,
\[
[x,y]=x^{-1}y^{-1}xy=x^{\bar y-1}. 
\]
Hence the result.
\end{proof}

\begin{corollary}\label{cor:normalaugment}
The augmentation ideal $\IGnm\X$ is a normal subgroup of $\Ynm$.
\end{corollary}
\begin{proof}
As we have seen in Lemma~\ref{lem:normalized}, $[\X,\Ynm]$ is normal in $\Ynm$. Hence, the result follows from Lemma~\ref{lem:augmentcommut}.
\end{proof}

The augmentation ideal allows us to describe the commutator subgroup of $\Ynm$ as follows.

\begin{prop}\label{prop:commutator}
 We have the equality
\[
[\Ynm,\Ynm]={\left\langle\IGnm\X,[\Ipnm,\Ipnm]\right\rangle}.
\]
\end{prop}
\begin{proof}
Recall that $\X$ is normal in $\Ynm$, $\Ynm/\X\cong\Gnm$ and $\Ipnm\X/\X\cong\Gnm$. Hence, every element of $\Ynm$ can be written as $x\cdot b$ for some $x\in \X$ and $b\in \Ipnm$. Let $x_1b_1,x_2b_2$ be any two elements of $\Ynm$ written in this way. We have the commutator identity
\[
[x_1b_1,x_2b_2]=[x_1,x_2b_2]^{b_1}[b_1,x_2b_2]=[x_1,x_2b_2]^{b_1}[b_1,b_2][b_1,x_2]^{b_2}.
\]
On the one hand, $[b_1,b_2]\in[\Ipnm,\Ipnm]$ by definition. On the other hand, both $[x_1,x_2b_2]$ and $[b_1,x_2]$ are inside $[\X,\Ynm]$, which is equal to $\IGnm\X$ by Lemma~\ref{lem:augmentcommut}. Hence the result.
\end{proof}

\begin{corollary}\label{cor:describeBnm}
We have
\begin{align*}
\Cnm&={\left\langle\IGnm\X,\Inm^\sigma:\nu\in\Sigma,\sigma\in \X\rtimes H\right\rangle};\\
\Bnm&=\IGnm\X+{\langle\Inm^\sigma:\nu\in\Sigma,\sigma\in\X\rtimes H\rangle}\cap \X.
\end{align*}
\end{corollary}
\begin{proof}By definition $[\Ynm,\Ynm]\subset\Cnm$ and $\IGnm\X\subset \X$, so we see from Proposition~\ref{prop:commutator} that
\[
\IGnm\X\subset \Bnm.
\]
Furthermore, Corollary~\ref{cor:normalaugment} says that $\IGnm\X$ is  normal in $\Ynm$. Therefore, the second equality follows from the first.

Recall that $\Cnm$ is defined to be
\[
{\left\langle[\Ynm,\Ynm],\Inm^\sigma:\nu\in\Sigma,\sigma\in\X\rtimes H\right\rangle}.
\]
Therefore, the first equality follows from the description of $[\Ynm,\Ynm]$ in Proposition~\ref{prop:commutator} and the fact that $[\Ipnm,\Ipnm]$ is contained in $\Ipnm$.
\end{proof}

\begin{prop}\label{prop:describeBnm}
The quotient $\Bnm/\IGnm\X$ is a $\Lambda(H)$-module  generated by the elements $x\in \X$ satisfying the property that $(x,h)\in \Inm$ for some $\nu\in\Sigma\setminus\{\p\}$ and $h\in H_m$.
\end{prop}
\begin{proof}
Suppose that $(x,h)\in\Inm$ for some $h\in H_m$ and $\nu\in\Sigma\setminus\{\p\}$, then $(1,h)\in \Ipnm$ thanks to \eqref{eq:Ip}.  Consequently, $(x,1)=(x,h)(1,h)^{-1}\in \Cnm\cap \X=\Bnm$.

Recall from Lemma~\ref{lem:normalized} that the image of $\Ipnm$ in $\Ynm/[\X,\Ynm]$ is a normal subgroup. By Lemma~\ref{lem:augmentcommut}, we have $[\Xnm,\Ynm]=\IGnm\X$. Hence,
\[{\langle\IGnm\X,\Ipnm\rangle}/\IGnm\X\unlhd\Cnm/\IGnm\X.\]

On applying Corollary~\ref{cor:describeBnm}, we deduce that every element in $\Cnm/\IGnm\X$ may be written as a product $\alpha\beta$ for some  $\alpha\in{\langle\Inm^\sigma:\nu\in\Sigma\setminus\{\p\},\sigma\in\X\rtimes H\rangle}$ and $\beta\in\Ipnm$.

Suppose that an element $\alpha\beta$ as above is contained in $\Bnm/\IGnm\X$. Then, $\alpha$ is a product of elements of the form $(x_\nu,h_\nu)^\sigma\in \Inm$, where $\sigma=(x_\sigma,h_\sigma)\in\X\rtimes H$. We have in fact
\[
(x_\nu,h_\nu)^\sigma=(x_\sigma^{(h_\nu-1)h_\sigma^{-1}}x_\nu^{h_\sigma^{-1}},h_\nu)
\]
given that $H$ is abelian. But $x_\sigma^{(h_\nu-1)h_\sigma^{-1}}\in [\X,\X\rtimes H_m]\subset \IGnm\X$ by  Lemma~\ref{lem:augmentcommut} and the fact that $h_\nu\in H_m$.
Furthermore, we have the identity $(x_\nu,h_\nu)(x_{\nu'},h_{\nu'})=(x_\nu x_{\nu'}^{h_\nu},h_\nu h_{\nu'})$, which implies that $\alpha=(x,h)$ for some $x\in\X$ is inside the $\Lambda(H)$-module generated by the elements $x_\nu$ as described in the statement of the proposition and $h\in H_m$ with $\beta=(1,h^{-1})$. Hence the result.
\end{proof}

\subsection{The $\Zp$-rank of $\Bnm/\IGnm\X$}

In the previous section, we showed in Proposition~\ref{prop:describeBnm} that we may find explicit generators for the quotient
$\Bnm/\IGnm\X$. We shall now bound its $\Zp$-rank.

The aforementioned quotient is generated by the ``projection" of $\Inm^h$ in $\X$, where $\nu\in\Sigma\setminus\{\p\}$ and $h\in H$. But the map $\Y\rightarrow \X$, $(x,g)\mapsto x$ is not a group homomorphism a priori. However, if $(x,g),(y,h)\in \Ynm$, we have
\[
(x,g)\cdot (y,h)=(xy^g,gh),
\]
and
\[
xy^g\equiv xy\mod [\X,\Ynm]=\IGnm\X
\]
by Lemma~\ref{lem:augmentcommut}. Therefore, 
the map 
\begin{align*}
\Ynm/\IGnm\X&\rightarrow \X/\IGnm\X\\
(x,g)&\mapsto x
\end{align*}
is a well-defined group homomorphism.

\begin{lemma}\label{lem:Zprank}
The quotient $\Bnm/\IGnm\X$ is a finitely generated $\Zp$-module. Furthermore, its rank is bounded by $r_{n,m}$, where $r_{n,m}$ is  the number of places of $\Knm$ above $\Sigma\setminus\{\p\}$.
\end{lemma}
\begin{proof}
As discussed above, the quotient is generated by the projections of $\Inm^h$ in $\X$, where $\nu\in\Sigma\setminus\{\p\}$ and $h\in H$, which corresponds to all the inertia groups of the places of $\M$ lying above $\Sigma\setminus\{\p\}$. 

If two primes of $\M$ differ by an element in $\Ynm$, then their inertia groups coincide modulo $\IGnm\X$ as we have seen in the proof of Proposition~\ref{prop:describeBnm}. Therefore, if for each prime of $\Knm$ lying above $\Sigma\setminus\{\p\}$, we pick one prime in $\M$ lying above this prime, the resulting inertia groups generate the quotient $\Bnm/\IGnm\X$.

Our result then follows from the fact that each of these inertia groups has $\Zp$-rank at most $1$. Indeed, all primes in $\Sigma\setminus\{\p\}$ are coprime to $p$ by assumption, so the maximal pro-$p$ extension of $K_\nu$ is isomorphic to $\Zp\rtimes\Zp$, which is of dimension $2$, as given by \cite[II.\S5.6 Exercices]{serre62}. Since $K_\nu$ admits  a one-dimensional unramified $\Zp$-extension, the inertia group  has dimension at most 1.
\end{proof}

\begin{lemma}\label{lem:numberofprimes}
Let $r_{n,m}$ be as defined in Lemma~\ref{lem:Zprank}, then 
\begin{enumerate}[(i)]
\item For $n$ sufficiently large, $r_{n,m}$ depends only on $m$;
\item $r_{n,m}=O(p^{(d-2)m})$ for $m\gg0$.
\end{enumerate}
\end{lemma}
\begin{proof}
Since  there is a finite number of primes in $K_{\infty,m}$ lying above each prime of $\Sigma\setminus\{\p\}$,  part (i) follows.

We now prove part (ii). Fix a prime $\bar\nu$ of $\Kinf$ above $\nu\in\Sigma\setminus\{\p\}$. As we have seen in the proof of Lemma~\ref{lem:Zprank}, the inertia group of $\bar\nu$ is a $p$-adic Lie group of dimension one. Furthermore, $\nu$ is inert over $K^\c/K_{n,0}$ for $n$ sufficiently large. Therefore, the decomposition group of $\bar\nu$ is of dimension two.

Let $\vpi$ be a prime of $\Knm$ above $\nu$. Let $G_{\vpi}$ be the decomposition group of  $\vpi$ in the extension $\Knm/K$. Our observation on the dimension of the decomposition group of $\bar\nu$ tells us that there exists a constant $C>0$ such that $|G_{\vpi}|\ge C\times p^{n+m}$ for all $\vpi$. But
\[
p^{n+(d-1)m}=|G:\Gnm|=\sum_{\vpi|\nu}|G_{\vpi}|.
\]
 If $r_{n,m,\nu}$ denotes the number of places of $\Knm$ above $\nu$. Then,
 $$r_{n,m,\nu}\le p^{n+(d-1)m}/Cp^{n+m}=p^{(d-2)m}/C,$$ which gives (ii).
\end{proof}

On combining these two lemmas, we deduce:

\begin{corollary}\label{cor:rank}
The quotient $\Bnm/\IGnm\X$ is a finitely generated $\Zp$-module with
\[
\rank_{\Zp}\Bnm/\IGnm\X=O(p^{(d-2)m})
\]
for $m\gg0$ (and independent of $n$).
\end{corollary}

\subsection{Algebraic structure of $\X$}
Our analysis on $\Bnm$ allows us to study the structure of $\X$ as a $\Lambda(G)$-module. In particular, we prove the following.

\begin{theorem}\label{thm:torsion}
The $\Lambda(G)$-module $\X$  is torsion.
\end{theorem}
\begin{proof}
As we have recalled above,  $\X$ is finitely generated over $\Lambda(G)$ by \cite[Proposition~3.1]{perbet}. In particular, if $\rho$ denotes its rank, \cite[Theorem~1.10]{harris00} tells us that
\[
\rank_{\Zp}\X/I_{G_{n,n} }\X=\rho \times p^{dn}+O(p^{(d-1)n}).
\]
By \eqref{eq:SES}, together with Corollary~\ref{cor:rank} and the finiteness of $\X_{n,n}$, we have in fact $$\rank_{\Zp}\X/I_{G_{n,n} }\X=O(p^{(d-2)n}).$$ This implies that $\rho=0$ and hence the result.
\end{proof}
This allows us to eliminate the most dominant term of  Perbet's formula \eqref{eq:perbet}.

\begin{corollary}\label{cor:perbet}
Let $\tenn$ denote the $p$-exponent of $\#\X_{n,n}/p^n$. Then,
\[
\tenn= \mu\times p^{dn}+O(np^{(d-1)n}),
\]
for some integer $\mu$.
\end{corollary}

Under an additional  hypothesis on $\X_{\infty,0}$, we can in fact show more:
\begin{theorem}\label{thm:MHG}
The $\Lambda(\Gamma)$-module $\X_{\infty,0}$ is finite if and only if $\X$ is finitely generated over $\Lambda(H)$. In particular, when this holds, $\X$ belongs to the $\MHG$-category.
\end{theorem}
\begin{proof}
The short exact sequence \eqref{eq:SES} becomes
 \[
0\rightarrow B_{\infty,0}/I_{H}\X \rightarrow\X/I_{H}\X\rightarrow\X_{\infty,0}\rightarrow0
\]
if we take $m=0$ and $n=\infty$. Corollary~\ref{cor:rank} tells us that  the first term of the short exact sequence is finite over $\Zp$. Therefore, the second term is finite over $\Zp$ if and only the  last term is. Suppose that $\X$ is finite over $\Lambda(H)$, then $\X/I_H\X$ is finite over $\Zp$, which gives one implication of the theorem. If on the other hand $\X_{\infty,0}$ is finite over $\Zp$, then so is $\X/I_H\X$.  Consequently, Nakayama's Lemma (c.f. \cite[Lemma~2.6]{coateshowson} or \cite{balisterhowson}) implies that $\X$ is finite over $\Lambda(H)$, which gives the other implication as claimed.
\end{proof}

\section{Growth in the $H$-direction}\label{S:estimateH}
In this section, we fix an integer $n\ge0$ and estimate the growth in $e_{n,m}$ as $m\rightarrow \infty$. Our strategy is to make use of our estimation on $\Bnm/\IGnm$ from \S\ref{S:estimateBnm}, in conjunction with the short exact sequence \eqref{eq:SES}.

Recall that $\X$ is finitely generated over $\Lambda(G)$. Consequently, $\X_{\Gamma_n}$ is a finitely generated $\Lambda(H)$-module. In fact, we can say more:

\begin{lemma}\label{lem:torsion}
The $\Lambda(H)$-module $\X_{\Gamma_n}$ is torsion.
\end{lemma}
\begin{proof}
Let $M$ be a finitely generated $\Lambda(H)$-module. If $\rank_{\Lambda(H)}M=r$, then 
\[
\rank_{\Zp}M_{H_m}= r\times p^{(d-1)m}+O(p^{(d-2)m})
\]
(c.f. \cite[Theorem~1.10]{harris00}).

The $H_m$-coinvariant of $\X_{\Gamma_n}$ is nothing but $\X/\IGnm\X$.  Since $\Xnm$ is finite, \eqref{eq:SES} tells us that $\X/\IGnm\X$ has the same $\Zp$-rank as $\Bnm/\IGnm\X$.  
Hence the result by Corollary~\ref{cor:rank}.
\end{proof}

We recall the following definition from \cite[\S4]{cuocomonsky}. Let $M$ be a finitely generated torsion $\Lambda(H)$-module. A structure $\cS$ on $M$ consists of a fixed integer $m_0$ together with a finite set of pairs $(\tau_i,M_i)$, where $\tau_i\in H\setminus H_1$ and $M_i$ submodules of $M$. For every structure of $M$, we define for $m\ge m_0$ $$A_m(\cS)=I_{H_m}M+\sum_i\Phi_{m/m_0}(\tau_i)\cdot M_i,$$
where $\Phi_{m/m_0}(X)$ denotes the polynomial $(X^{p^m}-1)/(X^{p^{m_0}}-1)$. Such a structure is said to be \textit{admissible} if $\rank_{\Zp}M/A_m(\cS)=O(p^{(d-3)n})$ (for $d\ge 3$) or $\rank_{\Zp}M/A_m(\cS)=O(1)$ (for $d=2$).

 Let $M$ be a finitely generated $\Zp$-module, we shall write $M_t$ for the torsion submodule of $M$ and $e(M)$ for the $p$-exponent of the order of $M_t$. The following result is proved in \textit{ loc. cit.}

\begin{theorem}\label{thm:CM}
Let $M$ be a finitely generated torsion $\Lambda(H)$-module and $\cS$ an admissible structure on $M$. Then,
\[
e(M/A_m(\cS)) =\mu_H(M)\times p^{(d-1)m}+\lambda_H(M)\times mp^{(d-2)m}+O(p^{(d-2)m})
\]
for some non-negative integers $\mu(M)$ and $\lambda(M)$ that are independent of $m$ and $\cS$.
\end{theorem}
\begin{proof}
This is Lemma~4.9 and Theorem~4.13 in \textit{op. cit.} when $d\ge3$. For the case $d=2$, we have $H=\Zp$ and the result follows from the classical results of \cite{iwasawa}.
\end{proof}

\begin{lemma}\label{lem:structure}
There exists an admissible structure $\cS$ on $\X_{\Gamma_n}$ such that $\Xnm=\X_{\Gamma_n}/A_m(\cS)$ for $m\gg0$.
\end{lemma}
\begin{proof}
Let $\nu\in\Sigma\setminus\{\p\}$. Then $\bInm$ is a subgroup of $H_m$. Therefore, there exists an integer $m_0$ such that $H_m/\bInm$ is torsion-free for all $m\ge m_0$. Since $\Sigma$ is finite, we may assume that $m_0$ is an integer satisfying this property for all $\nu$.

Recall from the proof of Lemma~\ref{lem:Zprank} that each $\bInm$ is of dimension $1$. Suppose that $I_{\nu_{n,m_0}}={\langle(x_{\nu},h_{\nu})\rangle}$. Then $h_\nu\in H_{m_0}\setminus H_{m_0+1}$. In particular, we may write $h_\nu=k_\nu^{p^{m_0}}$ for some $k_\nu\in H\setminus H_1$. Furthermore, for all $m\ge m_0$, we have $$I_{\nu_{n,m}}={\langle(x_{\nu},h_{\nu})^{p^{m-m_0}}\rangle}={\langle(x_{\nu}^{\Phi_{m/m_0}(k_\nu)},k_\nu^{p^{m}})\rangle}.$$
Therefore, Proposition~\ref{prop:describeBnm} tells us that
\[
\Bnm=\IGnm\X+\sum_{\nu}\Phi_{m/m_0}(k_{\nu})\Lambda(H)\cdot x_{\nu}.
\]
Hence, if we take $\cS=\{m_0,(k_\nu,\Lambda(H)\cdot x_\nu):\nu\in\Sigma\setminus\{\p\}\}$, then $\Bnm=A_m(\cS)$ as $\Xnm\cong\X/\Bnm$ by Lemma~\ref{lem:CFT}. Finally, the structure is admissible because $\Xnm$ is finite by definition.
\end{proof}

\begin{corollary}\label{cor:growthH}
For a fixed $n$, we have the formula
\[
\enm= \mu_H(\X_{\Gamma_n})\times p^{(d-1)m}+\lambda_H(\X_{\Gamma_n})\times mp^{(d-2)m}+O(p^{(d-2)m}).
\]
\end{corollary}
\begin{proof}
This follows from combining Theorem~\ref{thm:CM} with Lemmas~\ref{lem:torsion} and~\ref{lem:structure}.
\end{proof}

\section{Interlude: review on $\Lambda(\Gamma)$-modules}\label{S:reviewGamma}
We identify $\Lambda(\Gamma)$ with the power series ring $\Zp[[X]]$ on choosing a topological generator $\gamma$ of $\Gamma$ and identifying $\gamma-1$ with $X$. We write $\omega_n=(1+X)^{p^n}-1$ for $n\ge0$ and $\Phi_n=\omega_n/\omega_{n-1}$ denotes the cyclotomic polynomial of order  $p^n$ in $1+X$ for $n\ge1$. We shall fix a primitive $p^n$-th root of unity $\zeta_{p^n}$ and write $\epsilon_n=\zeta_{p^n}-1$. Finally, we write $\Phi_{n/n_0}=\omega_n/\omega_{n_0}$ for $n\ge n_0$ as in \S\ref{S:estimateH}.

Let $F\in\Lambda(\Gamma)$. Weierstrass Preparation Theorem tells us that there exists a factorization $F=u\times p^\mu\times g$, where $u\in \Lambda(\Gamma)^\times$, $\mu\in\ZZ_{\ge 0}$ and $g$ is a  distinguished polynomial. We shall write $\mu_\Gamma(F)=\mu$ and $\lambda_\Gamma(F)=\deg(g)$.

If $M$ is a finitely generated torsion $\Lambda(\Gamma)$-module, it is known that there exist $F_1,\ldots,F_r\in\Lambda(\Gamma)$ and an injective $\Lambda(\Gamma)$-morphism
\[
\phi:M/M'\rightarrow \bigoplus_{i=1}^r\Lambda(\Gamma)/(F_i),
\]
where $M'$ denotes the maximal pseudo-null $\Lambda(\Gamma)$-submodule of $M$ and 
 the  cokernel of $\phi$ is pseudo-null.   Note that a pseudo-null $\Lambda(\Gamma)$-module is simply a module over $\Lambda(\Gamma)$ with finite cardinality. 

The $\Lambda(\Gamma)$-ideal generated by the product $\prod_{i=1}^rF_i$ is called the characteristic ideal of $M$.
We write $\mu_\Gamma(M)=\sum\mu_\Gamma(F_i)$ and $\lambda_\Gamma(M)=\sum\lambda_\Gamma(F_i)$. We remark that the condition $\X_{\infty,0}$ being finitely generated over $\Zp$ in Theorem~\ref{thm:MHG} is equivalent to $\mu_\Gamma(\X_{\infty,0})=0$.

The following result of Iwasawa in \cite{iwasawa} is well-known.
\begin{theorem}\label{thm:iwasawa}
Let  $M$  be a finitely generated torsion $\Lambda(\Gamma)$-module.  Then, there exist constants $\nu_\Gamma(M)$ and $n_0$ such that $M_{\Phi_{n/n_0}}$ is finite with
\[
e(M_{\Phi_{n/n_0}})=\mu_\Gamma(M)\times p^n+\lambda_\Gamma(M)\times n+\nu_\Gamma(M)
\]
for all $n\ge n_0$.
\end{theorem}

This result has been reproved in many different places, e.g. \cite[\S10.2]{kobayashi03}, \cite[\S5.3]{neukirch} and \cite[\S13.3]{washington}. We shall give a sketch proof in the special case where  the characteristic ideal of $M$ is coprime to $\omega_n$ for all $n$. In doing so, we shall be able to say how large $n$ needs to be to ensure that the formula for $e(M_{\Phi_{n/n_0}})$ holds and give information on $\nu_\Gamma(M)$.

\begin{lemma}\label{lem:kob}
Let $n\ge1$ and $F\in \Lambda(\Gamma)$ with $\gcd(F,\omega_{n})=1$. Consider the projection map
\[
\pi_n:\Lambda(\Gamma)/(F,\omega_n)\rightarrow\Lambda(\Gamma)/(F,\omega_{n-1}).
\]
We have 
\begin{enumerate}[(i)]
\item $\rank_{\Zp}\Lambda(\Gamma)/(F,\omega_{n})=\rank_{\Zp}\Lambda(\Gamma)/(F,\omega_{n-1})=0$;
\item  $\ker\pi_n$ is finite and   $e(\ker\pi_n)=\ord_{\epsilon_n}F(\epsilon_n)$.
\end{enumerate}
\end{lemma}
\begin{proof}
This is well-known. See for example \cite[\S13.3]{washington} or  \cite[Lemma~10.5]{kobayashi03}. 
\end{proof}

\begin{corollary}\label{cor:kernel}
Under the same notation as Lemma~\ref{lem:kob}, if $$F=u\times p^\mu\times\prod_{i=1}^rF_i,$$ where $u\in\Lambda(\Gamma)^\times$, $\mu=\mu_\Gamma(F)$ and $F_i$ are distinguished polynomials of degree $d_i$, then 
\[
e(\ker\pi_n)=\mu\times p^{n-1}(p-1)+\lambda_\Gamma(F)
\]
whenever $p^{n-1}(p-1)> d_i$ for $i=1,\ldots,r$.
\end{corollary}
\begin{proof}
 Firstly, it is immediate that $\ord_{\epsilon_n}(u)=0$ and $\ord_{\epsilon_n}p^\mu=\mu\times p^{n-1}(p-1)$. Secondly, for each $i$, we may write  $F_i$ as $X^{d_i}+pG_i$ for some polynomial $G_i$ defined over $\Zp$ with degree $<d_i$. As $$d_i=\ord_{\epsilon_n}(\epsilon_n^{d_i})<p^{n-1}(p-1)\le\ord_{\epsilon_n}(pG_i(\epsilon_n)),$$
we have $\ord_{\epsilon_n}F_i(\epsilon_n)=\deg F_i$. Hence the result.
\end{proof}

\begin{corollary}\label{cor:cyclicformula}
Suppose that $F$ is as in Corollary~\ref{cor:kernel}, then
\[
e(\Lambda(\Gamma)/(F,\omega_n))-e(\Lambda(\Gamma)/(F,\omega_{n_0}))=\mu\times (p^n-p^{n_0})+\lambda_\Gamma(F)\times (n-n_0)
\]
for all $n\ge n_0$, where $n_0$ is a fixed integer satisfying $p^{n_0-1}(p-1)>d_i$ for $i=1,\ldots, r$.
\end{corollary}
\begin{proof}
For $m\in\{n,n-1,\ldots n_0+1\}$, Lemma~\ref{lem:kob}(ii) tells us that
\[
e(\ker\pi_m)=e(\Lambda(\Gamma)/(F,\omega_m))-e(\Lambda(\Gamma)/(F,\omega_{m-1})).
\]
Hence the result by Corollary~\ref{cor:kernel}.
\end{proof}

\begin{lemma}\label{lem:pseudosame}
Let $M$ be a finitely generated torsion $\Lambda(\Gamma)$-module, with maximal pseudo-null submodule $M'$. Let $\phi:M/M'\rightarrow \Lambda/(F)$ be an injective $\Lambda(\Gamma)$-morphism with finite cokernel. Suppose that $\gcd(F,\omega_m)=1$ for all $m\ge1$. Then,  $M_{\Gamma_n}$ is finite and
\[
e(M_{\Gamma_n})=e(\Lambda/(F,\omega_n))+e(M'_{\Gamma_n})
\]
for any integer $n\ge1$.
\end{lemma}
\begin{proof}
Let $C$ be the cokernel of $\phi$.  Our assumption on $F$ implies that 
\[
\Lambda(\Gamma)/(F)\stackrel{\omega_n}{\longrightarrow} \Lambda(\Gamma)/(F)
\]
is injective. On applying the snake lemma to the short exact sequence $0\rightarrow M/M'\rightarrow  \Lambda/(F) \rightarrow C\rightarrow0$, we have the exact sequence
\[
0\rightarrow C^{\Gamma_n}\rightarrow(M/M')_{\Gamma_n}\rightarrow \Lambda/(F,\omega_n) \rightarrow C_{\Gamma_n}\rightarrow 0.
\]
Since $C$ is finite, the first and the last terms of the exact sequence have the same cardinality. Since $F$ is coprime to $\omega_n$, $\Lambda/(F,\omega_n)$ is finite and hence have the same cardinality as $(M/M')_{\Gamma_n}$.

Since $M/M'$ injects into $\Lambda(\Gamma)/(F)$, the fact that multiplication by $\omega_n$ is injective on $\Lambda(\Gamma)/(F)$ means that it is also injective on $M/M'$.  Therefore, if we apply the snake lemma to $0\rightarrow M'\rightarrow M\rightarrow M/M'\rightarrow0$, we have
\[
e(M_{\Gamma_n})=e((M/M')_{\Gamma_n})+e((M')_{\Gamma_n}),
\]
which implies the result. 
\end{proof}

We note that  $M'_{\Gamma_n}=M'$ for $n\gg0$ (see for example \cite[Lemma~5.3.14(v)]{neukirch}).

\begin{prop}\label{prop:gammagrowth}
Let $M$ be a finitely generated $\Lambda(\Gamma)$-module. Let $F\in\Lambda(\Gamma)$ be a generator of its characteristic ideal.  Suppose that  $\gcd(F,\omega_m)=1$ for all integers $m\ge1$. Let  $n_0$ be an integer such that all the irreducible distinguished polynomials that divide  $F$ have degree $<p^{n_0-1}(p-1)$. Then, $$e(M_{\Gamma_n})-e(M_{\Gamma_{n_0}})=\mu_\Gamma(M)\times(p^n-p^{n_0})+\lambda_\Gamma(M)\times(n-n_0)+e(M'_{\Gamma_n})-e(M'_{\Gamma_{n_0}})$$
for all $n\ge n_0$.
\end{prop}
\begin{proof}
This is an immediate consequence of Corollary~\ref{cor:cyclicformula} and Lemma~\ref{lem:pseudosame}.
\end{proof}

\section{Estimating the growth of $\enn$ when $d=2$}\label{S:enn}
Throughout this section, we assume that $d=2$. Let $m,n\ge0$ be integers and consider the $\Lambda(\Gamma)$-module
\[
M_m:=\X_{\infty,m}\cong\X/B_{\infty,m},
\]
where $B_{\infty,m}$ is as defined in \S\ref{S:estimateBnm}. By definition, this is the Galois group of the maximal pro-$p$ unramified extension of $K_{\infty,m}$.
Then, on taking $\Gamma_n$-coinvariant, we have
\[
(M_m)_{\Gamma_n}=\X/\langle I_{\Gamma_n}\X,B_{\infty,m}\rangle\cong \Xnm
\]
thanks to Lemma~\ref{lem:CFT}. As $\Xnm$ is finite,  $M_m$ is  a finitely generated $\Lambda(\Gamma)$-module whose characteristic ideal is coprime to $\omega_n$ for all $n\ge1$. 
We deduce from Proposition~\ref{prop:gammagrowth}  that
for a fixed $m$, there exists an integer $n_m$ such that for all  $n\ge n_{m}$,
\begin{equation}\label{eq:gammagrowth}
\enm-e_{n_{m},m}=\mu_{\Gamma}(M_m)\times(p^n- p^{n_{m}})+\lambda_\Gamma(M_m)\times (n-n_{m})+e'_{n,m}-e_{n_m,m}',
\end{equation}
where $e_{n,m}'=e((M_m')_{\Gamma_n})$, with $M_{m}'$ being the maximal pseudo-null submodule of $M_m$. We shall study how  $\lambda_\Gamma(M_m)$, $\mu_\Gamma(M_m)$, $e(M_{m}')$ and  $n_m$ vary in $m$.

\subsection{Estimating Iwasawa invariants}
 In this section, we assume that $\X\in\MHG$, where $\MHG$ is the category as defined in the introduction. 
Let us recall the definition of $\mu$-invariants of finitely generated $\Lambda(G)$-modules.
Let $M$ be a finitely generated $\Lambda(G)$-module that is $\Zp$-torsion. It is proved in \cite{venjakob02} that $M$ is pseudo-isomorphic to
\[
\bigoplus_i \Lambda(G)/p^{n_i}
\]
for some integers $n_i$. We have the $\mu$-invariant $\mu_G(M):=\sum n_i$. More generally, if $M$ is a finitely generated $\Lambda(G)$-module. We define $\mu_G(M):=\mu_G(M(p))$.

We shall write $\tX$ for the quotient $\X/\X(p)$, which is finitely generated over $\Lambda(H)$ by our $\MHG$-hypothesis. Let $\tau_\X$ denote the $\Lambda(H)$-rank of $\tX$. We have the following short exact sequence
\begin{equation}\label{eq:XP}
0\rightarrow \X(p)\rightarrow \X\rightarrow \tX\rightarrow 0.
\end{equation}

\begin{prop}\label{prop:growthlambdamu}
We have
\[
\lambda_\Gamma(\X_{H_m})=\tau_\X\times p^{(d-1)m}+O(1),\quad \mu_\Gamma(\X_{H_m})=\mu_G(\X)\times p^{m}
\]
\end{prop}
\begin{proof}
From \eqref{eq:XP}, there is a long exact sequence
\[
H_1(H_m,\tX)\rightarrow \X(p)_{H_m}\rightarrow \X_{H_m}\rightarrow\tX_{H_m}\rightarrow 0.
\]
Since $\X(p)$ is $\Zp$-torsion, this tells us that
\[
\rank_{\Zp}\X_{H_m}=\rank_{\Zp}\tX_{H_m}.
\]
But the latter is equal to $\tau_\X\times p^{m}+O(1)$ as given by \cite[Theorem~1.10]{harris00}. This gives the formula for $\lambda_\Gamma(H_m)$.

We now turn our attention to the $\mu$-invariant. Since $\tX$ is finitely generated over $\Lambda(H)$ and hence over $\Lambda(H_m)$, the homology groups $H_i(H_m,\tX)$ are finitely generated over $\Zp$ for all $i\ge0$. Therefore, the same long exact sequence tells us that
\[
\mu_\Gamma(\X(p)_{H_m})=\mu_\Gamma(\X_{H_m}).
\] 
Following \cite[Lemma~5.2]{coateskim}, we have the equation
\[
\mu_{H_m\rtimes \Gamma}(\X)=\mu_\Gamma\left(\X(p)_{H_m}\right).
\]
But $[G:H_m\rtimes\Gamma]=p^m$,  so the formula \cite[(4)]{CSfine}) tells us that
\[
\mu_{H_m\rtimes \Gamma}(\X)=\mu_G(\X)\times p^m.
\]
 Hence the result follows on combining the last three equations.
\end{proof}

\subsection{Estimating maximal finite submodules and $\enn$}
In this section, we assume that the hypothesis $\mu_\Gamma(M_0)=0$ holds. We recall from Theorem~\ref{thm:MHG} that this is equivalent to  $\X$ being finitely generated over $\Lambda(H)$.
This allows us to deduce the following estimates.

\begin{prop}\label{prop:finitesub}
If $M_m'$ is the maximal finite $\Lambda(\Gamma)$-submodule of $M_m$, then
\[
 e(M_m')=O(p^m).
\]
\end{prop}
\begin{proof}
We recall from  \S\ref{S:estimateH} that there exist an integer $m_0$, $x_\nu\in\X$ and $k_\nu\in H\setminus H^p$ for each $\nu\in\Sigma\setminus\{\p\}$ such that
\[
B_{\infty,m}=I_{H_m}\X+\sum_\nu \Phi_{m/m_0}(k_\nu)\Lambda(H)\cdot x_\nu
\]
for all $m\ge m_0$. Since we are assuming that $d=2$ here, we may in fact assume that $k_\nu=h$ for all $\nu$, where $h$ is  some fixed topological generator of $H$.  In particular, we have the equation
\[
B_{\infty,m}=\Phi_{m/m_0}(h)B_{\infty,m_0}.
\]

Since we are assuming that $\X$ is finitely generated over $\Lambda(H)$ and $H\cong\Zp$, the structure theorem for finitely generated $\Lambda(H)$-modules tells us that
\[
B_{\infty,m_0}\sim \Lambda(H)^{r}\oplus T,
\]
where $\sim$ signifies a pseudo-isomorphism, $r=\rank_{\Lambda(H)}B_{\infty,m_0}$ and $T$ is a torsion  $\Lambda(H)$-module. Therefore,
\[
e(B_{\infty,m_0}/B_{\infty,m})=e((B_{\infty,m_0})_{\Phi_{m/m_0}(h)})=e(T_{\Phi_{m/m_0}(h)})=O(p^m),
\]
as given by Theorem~\ref{thm:iwasawa} (with $H$ replacing $\Gamma$).

The isomorphism theorem gives us the short exact sequence
\[
0\rightarrow B_{\infty,m_0}/B_{\infty,m}\rightarrow M_m\rightarrow M_{m_0}\rightarrow 0.
\]
Hence $e(M_m)\le e(B_{\infty,m_0}/B_{\infty,m})+e(M_{m_0})$, which finishes the proof.
\end{proof}

\begin{corollary}\label{cor:main}
If $\mu_\Gamma(M_0)=0$, then
\[
\enn=\tau_\X\times np^n+O(p^n).
\]
\end{corollary}
\begin{proof}
Under our assumption on $\X$, \cite[Corollary~A.4]{DLsha} tells us  that there exists an integer $\rho$ such that the $\Lambda(\Gamma)$-characteristic ideal of $\X_{H_m}$ factorises into polynomials whose degrees are bounded by $\rho$. The same can be said about $M_m$ given that it is a quotient of $\X_{H_m}$. In particular, by Proposition~\ref{prop:gammagrowth}, the estimates in \eqref{eq:gammagrowth} hold whenever $p^{n-1}(p-1)>\rho$. Hence, we may choose $n_m=n_{0}$ for some fixed $n_0$ that is independent of $m$. 

We recall from  Corollary~\ref{cor:growthH} that $e_{n_0,m}=O(p^m)$.
Furthermore, if $\X$ is finitely generated over $\Lambda(H)$, then $\mu_G(\X)=0$. Hence, our result follows on combining \eqref{eq:gammagrowth} with Propositions~\ref{prop:growthlambdamu} and \ref{prop:finitesub}.
\end{proof}

\section{Bounding $\tenn$ in the case $d=2$}
In this section, we assume that $d=2$ and $\X\in\MHG$. Since $X$ is torsion over $\MHG$ in this setting, we have already seen in Corollary~\ref{cor:perbet} that the asymptotic formula of Perbert can be  improved to 
\[
\tenn=\mu\times p^{2n}+O(np^n).
\]
However, the error term is larger than that of Corollary~\ref{cor:main}. We now show that we may obtain an upper bound on $\tenn$ with the same error term under our assumption  $\X\in\MHG$.

\begin{prop}\label{prop:estimateXGnn}
Assume that $\X\in\MHG$ and write $\tX=\X/\X(p)$. Then,
\[e(\X_{G_{n,n}})\le\mu_G(\X)\times p^{2n}+\tau_\X\times np^n+O(p^n),\]
where $\tau_\X=\rank_{\Lambda(H)}\tX$.
\end{prop}
\begin{proof}
From  \eqref{eq:XP}, we obtain the long exact sequence
\begin{equation}
\cdots \rightarrow \X(p)_{G_{n,n}}\rightarrow \X_{G_{n,n}}\rightarrow \tX_{G_{n,n}}\rightarrow 0.
\end{equation}
 We shall use $e(\tX_{G_{n,n}})$ and $e(\X(p)_{G_{n,n}})$ to bound $e(\X_{G_{n,n}})$.

Since $\tX$ is finitely generated over $\Lambda(H)$, we have already seen in the proofs of Propositions~\ref{prop:growthlambdamu} and \ref{prop:finitesub} that  $\tX_{H_n}$ is a finitely generated $\Lambda(\Gamma)$-module with 
\[
\mu_\Gamma(\tX_{H_n})=0,\quad \lambda_\Gamma(\tX_{H_n})=\tau_\X\times p^n+O(1), \quad e(\tX'_{H_n})=O(p^n).
\]
Consequently, $e(\tX_{G_{n,n}})=\tau_\X\times np^n+O(p^n)$ by Theorem~\ref{thm:iwasawa}.

Since $\X(p)$ is $\Zp$-torsion and finitely generated over $\Lambda(G)$, it follows that $\X(p)_{G_{n,n}}$ is finite. Recall that there is a pseudo-isomorphism of $\Lambda(G)$-modules
\[
\X(p)\sim\bigoplus_i \Lambda(G)/p^{n_i}
\]
for some integers $n_i$. In general, if $M$ and $N$ are pseudo-isomorphic $\Lambda(G)$-modules that are both $\Zp$-torsion, then \cite[Lemma~4.2]{DLsha} tells us that 
\[
\#M_{G_{n,n}}=\#N_{G_{n,n}}\times p^{O(p^{n})}
\]
under our assumptions. Therefore, 
\[
\#\X(p)_{G_{n,n}}=\#\bigoplus_i \Zp[G/G_{n,n}]/p^{n_i}\times p^{O(p^{n})}
\]
and hence
\[
e(\X(p)_{G_{n,n}})=p^{2n}\times\sum_i n_i+p^{n}=\mu_{G}(\X)\times p^{2n}+O(p^{n}).
\]
This finishes our proof.
\end{proof}

\begin{corollary}\label{cor:upper}
We have the upper bound
\[
\tenn\le \mu_G(\X)\times p^{2n}+\tau_\X\times np^n+O(p^n)
\]
\end{corollary}
\begin{proof}
First of all, we observe that $\tenn\le e(\X_{G_{n,n}})/p^n)$ thanks to the short exact sequence \eqref{eq:SES}. Therefore, it is enough to bound $e(\X_{G_{n,n}})/p^n)$.

Since $\X_{G_{n,n}}$ is a finitely generated $\Zp$-module, it is isomorphic to
\[
\Zp^{\oplus a_n}\oplus T_n
\]
for some integer $a_n\ge0$ and some finite $\Zp$-module $T_n$. This gives an isomorphism of abelian groups
\[
(\X_{G_{n,n}})/p^n\cong (\ZZ/p^n)^{a_n} \times T_n/p^n.
\]
In particular, this tells us that
\[
e((\X_{G_{n,n}})/p^n)=a_n\times n+e(T_n/p^n)\le a_n\times n+e(T_n).
\]
Since we are assuming $d=2$, Corollary~\ref{cor:rank} tells us that $a_n=O(1)$. Hence we are done by the bound on $e(T_n)$  given in Proposition~\ref{prop:estimateXGnn}.
\end{proof}

\bibliographystyle{amsalpha}
\bibliography{references}

\end{document}